\documentclass[a4paper,11pt]{article}
\usepackage[margin=0pt]{geometry} 

\usepackage{slashbox} 
\newcommand{\verteq}{\rotatebox{90}{$\,=$}} 

\usepackage{pgfplots}
\pgfplotsset{compat=1.15}
\usepackage{mathrsfs}
\usetikzlibrary{arrows}

\usepackage{pstricks,pst-node}

\usetikzlibrary{shapes.misc}
\tikzset{cross/.style={cross out, draw=black, fill=none, minimum size=2*(#1-\pgflinewidth), inner sep=0pt, outer sep=0pt}, cross/.default={2pt}}

\usepackage[T1]{fontenc}

\usepackage[]{amsmath, amssymb, amsthm, tabularx}

\usepackage[]{a4, xcolor, here}

\usepackage[]{pstricks, pst-text, pst-node, pst-tree, gastex}

\usepackage{tikz}

\usepackage[tikz]{bclogo}

\usepackage{comment}

\usepackage{boites,graphicx}

\usepackage{soul}

\definecolor{pastelyellow}{rgb}{0.99, 0.99, 0.59}
\definecolor{aqua}{rgb}{0.0, 1.0, 1.0} 
\definecolor{aquamarine}{rgb}{0.5, 1.0, 0.83} 
\definecolor{bananayellow}{rgb}{1.0, 0.88, 0.21}
\definecolor{burgundy}{rgb}{0.5, 0.0, 0.13}
\definecolor{ao(english)}{rgb}{0.0, 0.5, 0.0}

\setul{}{0.2ex}
\setulcolor{bananayellow}

\usepackage[normalem]{ulem}

\usepackage{mathrsfs}

\usepackage{stmaryrd}

\usepackage{enumerate}

\usepackage{paralist}

\usepackage{multienum}

\usepackage{multicol}

\usepackage{fancyhdr}

\usepackage{datetime}

\usepackage{everypage}
\usepackage[contents={},opacity=1,scale=1.6,
color=gray!90]{background}
\usepackage{ifthen}

\usepackage[]{hyperref} 

\hypersetup{
	colorlinks = true,
	linkcolor = red,
	anchorcolor = black,
	citecolor = blue, 
	filecolor = cyan,
	menucolor = red,
	runcolor = cyan,
	urlcolor = blue,
	linkbordercolor = {white},
	linktocpage = true
}

\newtheorem{theorem}{Theorem}[section]
\newtheorem{proposition}[theorem]{Proposition}

\newtheorem{corollary}[theorem]{Corollary}

\theoremstyle{definition}
\newtheorem{definition}[theorem]{Definition}

\newtheorem{example}[theorem]{Example}
\newtheorem{remark}[theorem]{Remark}

\makeatletter
\def\thmhead@plain#1#2#3{%
	\thmname{#1}\thmnumber{\@ifnotempty{#1}{ }\@upn{#2}}%
	\thmnote{ {\the\thm@notefont#3}}}
\let\thmhead\thmhead@plain
\makeatother

\newcommand{\bd}{\mathbf{d}}

\newcommand{\cC}{\mathcal{C}}
\newcommand{\cD}{\mathcal{D}}
\newcommand{\cE}{\mathcal{E}}
\newcommand{\cF}{\mathcal{F}}
\newcommand{\cG}{\mathcal{G}}

\newcommand{\cU}{\mathcal{U}}
\newcommand{\cV}{\mathcal{V}}

\newcommand{\bbZ}{{\mathbb Z}} 
\newcommand{\bbF}{{\mathbb F}} 

\renewcommand{\geq}{\geqslant}
\renewcommand{\leq}{\leqslant}

\begin{document}

	\renewcommand{\headrulewidth}{0pt}
	
	\rhead{ }
	\chead{\scriptsize Motzkin numbers and flag codes}
	\lhead{ }

	\title{Motzkin numbers and flag codes}

	\author{\renewcommand\thefootnote{\arabic{footnote}}
		Clementa Alonso-Gonz\'alez\footnotemark[1] \ and  Miguel \'Angel Navarro-P\'erez\footnotemark[2]}

	\footnotetext[1]{Dpt.\ de Matem\`atiques, Universitat d'Alacant, 
	Sant Vicent del Raspeig, Ap.\ Correus 99, E -- 03080 Alacant.}
	\footnotetext[2]{Centro Universitario EDEM Escuela de Empresarios,  
	Muelle de la Aduana, s/n -- Valencia }

	{\small \date{\usdate{\today}}} 
	
	\maketitle
	
	\begin{abstract}
	
\emph{Motzkin numbers} have been widely studied since they count many different combinatorial objects. In this paper we present a new appearance of this remarkable sequence in the network coding setting through a particular case of \textit{multishot codes} called \textit{flag codes}. A flag code is a set of sequences of nested subspaces (\emph{flags}) of a vector space over the finite field $\bbF_q$. If the list of dimensions is $(1, \ldots, n-1)$, we speak about a \textit{full flag code}. The \emph{flag distance} is defined as the sum of the respective subspace distances and can be represented by means of the so-called \emph{distance vectors}. We show that the number of distance vectors corresponding to the full flag variety on $\bbF_q^n$ is exactly the $n$-th Motzkin number. Moreover, we can identify the integer sequence that counts the number of possible distance vectors associated to a full flag code with prescribed minimum distance.
	\end{abstract}
	
	\textbf{Keywords:} Motzkin numbers, Motzkin paths, flag codes, flag distance, distance vectors.
	

	\section{Introduction}\label{sec:Introduction}
Network coding is the most efficient way to send data across a network modelled as a directed acyclic multigraph with multiple senders and receivers. The key is that the intermediate nodes can perform random linear combinations of the incoming inputs. This improves considerably the information flow although it can also lead to error propagation and erasures (see \cite{AhlsCai00}). To solve this problem, in \cite{KoetKschi08} the authors consider subspaces of $\mathbb{F}_q^n$ as codewords and \emph{subspace codes} as collections of subspaces. In particular, if all the subspaces have the same dimension, we have \emph{constant dimension codes}. Consult \cite{TrautRosen18} to have an overview of the most representative works in this subject.

On the other side, subspace codes can be  considered as \emph{one-shot subspace codes} since sending a codeword (a subspace) requires just a channel use. This gives rise to the idea, first suggested in \cite{NobUcho2009}, of using the channel several times to get  \emph{multishot subspace codes}. Under this approach codewords are sequences of subspaces. In particular, it is possible to consider \emph{constant type flags}, that is, sequences of nested subspaces with fixed dimensions as it was proposed in \cite{LiebNebeVaz18}.  In this seminal paper, collections of flags are called \emph{flag codes} and they are presented as a generalization of constant dimension codes. The recent works \cite{cotas, CasoPlanar,CasoNoPlanar, Kurz20} deal with different questions related to the parameters and construction of flag codes.

In this paper we focus on the distance parameter associated to full flag codes, that is, those whose sequence of dimensions is $(1,\ldots, n-1)$. As usual in the multishot context, the distance between flags is computed as the sum of the subspace distances corresponding to each shot, which provokes that many different combinations of them can give the same flag distance value. This fact was carefully pointed out in \cite{cotas}, where the authors introduce the notion of \emph{distance vector} (associated to a given distance value) in order to complete the information concerning the distance of pairs of flags. In this framework, it naturally arises the question of knowing in how many different ways the distance between a couple of flags can be distributed. We provide the answer to this question by showing that the cardinality of the set of distance vectors corresponding to the full flag variety on $\bbF_q^n$ is given by the $n$-th Motzkin number. 
	
\textit{Motzkin numbers} have been widely investigated as they appear in a great variety of combinatorial objects (see \cite{Sloane, Stanley} for more details). Here we look at them in terms of certain lattice paths from $(0,0)$ to $(n,0)$ that consist of horizontal steps $(1,0)$, up steps $(1,1)$ and down steps $(1,-1)$ and never goes below the $x$-axis (\textit{Motzkin paths}). Thus, we exhibit a correspondence between distance vectors and Motzkin paths that results to be crucial for our purposes. This idea has been inspired by the combinatorial approach to flag codes developed in \cite{Combinatorial} and some interesting talks with Paulo Almeida and Alessandro Neri. 
	
The remain of the paper is organized as follows: in Section \ref{sec: preliminaries} we remember some basic definitions and results on Motzkin numbers, subspace codes and flag codes. In Section \ref{sec: distance vectors} we recall the concept of \emph{distance vector} associated to a couple of flags and we pose the problem of counting such objects. In Section \ref{sec: the bijection} we establish a bijection between the set of distance vectors associated to the full flag variety on $\bbF_q^n$ and the set of Motzkin paths with $n$ steps which permits to compute the cardinality of the former set as well as the one of some important subsets of it. Moreover, we show that such a bijection takes the flag distance associated to a distance vector to the area under the corresponding Motzkin path. Finally, in Section \ref{sec:conclusions} we propose some open questions that could be addressed taking into account our work.

\section{Preliminaries}\label{sec: preliminaries}
	In this section we briefly recall some basic background on  Motzkin numbers, subspace codes and flag codes.

	\subsection{Motzkin numbers}
Let us consider the set of lattice paths in $\bbZ^2$ whose permitted steps are the up diagonal step $(1,1)$ denoted by $U$, the down diagonal step $(1, -1)$, denoted by $D$, and the horizontal step $(1,0)$, denoted by $H$. Let $U(x,y)$ represent the set of all unrestricted lattice paths running from $(0,0)$ to $(x,y)$ and using the steps $U, D$ and $H$.
\begin{definition}
	A \emph{Motzkin path} of length $n$ is a lattice path in $U(n,0)$ that never runs below the $x$-axis. We denote by $\mathcal{M}_n$  the set of Motzkin paths of length $n$. A \emph{Dyck path} is a Motzkin path that does not contain horizontal steps. The set of Dyck paths of length $2n$ will be denoted by $\mathcal{D}_n$.
\end{definition}
It is clear that, if horizontal steps are not allowed, the length of a Motzkin path (a Dyck path, in this case) must be even. On the other hand, it is well known that the cardinality of the set $\mathcal{M}_n$ is the $n$-th Motzkin number $M_n$ (sequence $A001006$ in \cite{Sloane}) and  the cardinality of $\mathcal{D}_n$ is the $n$-th Catalan number $C_n$ (sequence $A000108$ in \cite{Sloane}). More precisely, we can compute $M_0=C_0=1$ and
$$M_n=M_{n-1}+\sum_{k=0}^{n-2}M_k M_{n-k-2}, \quad C_n=\sum_{k=0}^{n-1}C_k C_{n-k-1} \quad (n\geq 1). $$

\noindent  The Motzkin numbers sequence, whose first ten terms are 
 $$1, 1, 2, 4, 9, 21, 51, 127, 323, 835, $$
  was introduced by Theodor Motzkin \cite{Motzkin}, while counting possible sets of nonintersecting chords joining some of $n$ points on a circle. Since this seminal paper it has been the subject of numerous studies over the last forty years. See \cite{Aigner1998} and \cite{DonaSha1977} for a couple of the earliest surveys where the authors exhibit different combinatorial objects counted by this sequence, including Motzkin paths. 

A commonly used and concise way of representing a Motzkin path $p \in \mathcal{M}_n$ is by its corresponding Motzkin word (see \cite{Dukes}, for instance).
\begin{definition}
 A \textit{Motzkin word} of length $n$  is  a sequence of $n$ letters in the alphabet $\{U,D, H\}$, namely $p=p_1p_2 \dots p_n$,  with the constraint that the number of occurrences of the letter $U$ is equal to the number of occurrences of the letter $D$ and, for every $i \in \{1, \ldots, n\}$, the number of occurrences of $U$ in
the subword $p_1p_2 \dots p_i$ is not smaller than the one of $D$.
\end{definition}
 In the following, we will not distinguish between a Motzkin path and the corresponding Motzkin word. 
 	\begin{figure}[h]
 	\begin{center}
 		\includegraphics[scale=0.75]{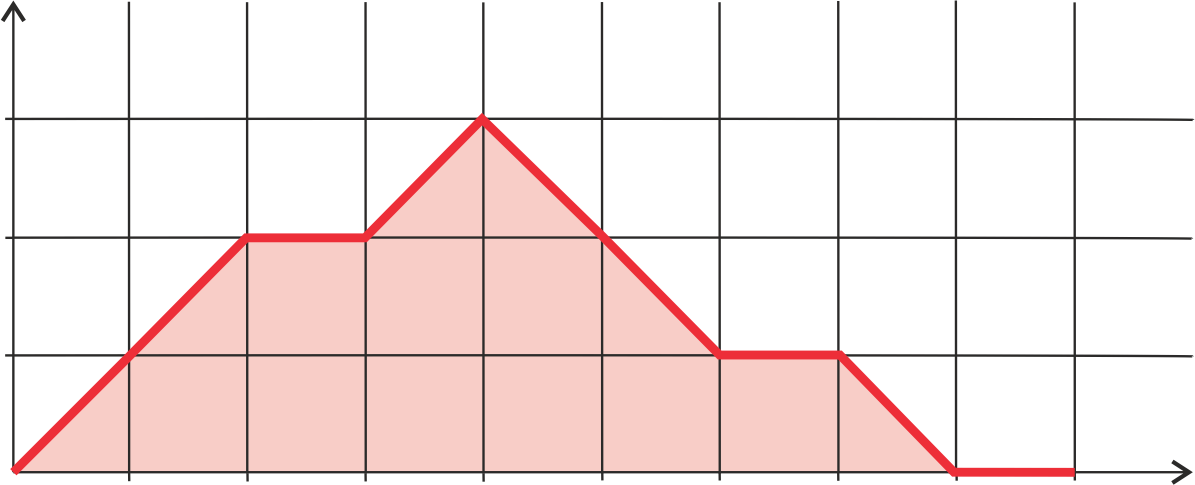}
 		\caption{The Motzkin path $p=UUHUDDHDH$ in $\mathcal{M}_9$.}
 		\label{Fig:Motzkin path 1}
 	\end{center}
 \end{figure}
\begin{remark}\label{rem: parenthesis}
	Observe that if we see the letters $U$ and $D$ as an open and closed parenthesis respectively, a Motzkin path is a word in $\{U,H, D\}$ such that the subword on the letters $U$ and $D$ forms a \textit{balanced parenthesization}, that is, the pairs of parantheses on them are correctly matched. This viewpoint will be useful for us in Section \ref{sec: the bijection}.
\end{remark}

\begin{example}
		To the Motzkin path  $p=UUHUDDHDH$ in Figure \ref{Fig:Motzkin path 1} it corresponds the parenthesization $((()))$.	
\end{example}

\begin{definition}
	The area of a Motzkin path $p$, denoted by $A(p)$, is the area of the region above the $x$-axis and below the path $p$.
\end{definition}
\begin{remark}
For the Motzkin path in Figure \ref{Fig:Motzkin path 1} we have $A(p)=12$.  It is clear that for any $p \in \mathcal{M}_n$, it holds $0 \leq A(p) \leq \lfloor \frac{n^4}{4} \rfloor$. In fact, if we denote by $\mathcal{M}_n(k)$ the set of Motzkin paths with length $n$ and area $k $ with $0 \leq k  \leq \lfloor \frac{n^4}{4} \rfloor$, the cardinality of that set, denoted $T(n,k)$, is given by the sequence $ A129181 $ in \cite{Sloane}. As far as we know, at the moment there is not a recurrence formula to calculate this sequence. We present here a table that covers it for values $n \in \{0,1, \dots, 8\}$.
\end{remark}

	\begin{table}[H]\label{Tab: T(n,k)}
		\centering
		\begin{footnotesize}
			\begin{tabular}{c|p{0.3cm}p{0.3cm}p{0.3cm}p{0.2cm}p{0.3cm}p{0.2cm}p{0.3cm}p{0.2cm}p{0.2cm}p{0.3cm}p{0.2cm}p{0.2cm}p{0.3cm}p{0.2cm}p{0.2cm}p{0.2cm}p{0.3cm}}
				\hline
				  \backslashbox{$n$}{$k$} & $\textbf{0}$       & $\textbf{1}$ &  $\textbf{2}$ &  $3$ & $\textbf{4}$ & $5$ & $\textbf{6}$ & $7$ & $8$ & $\textbf{9}$ & $10$ & $11$ & $\textbf{12}$ & $13$ & $14$ & $15$ & $\textbf{16}$    \\ 
				\hline
				0               & $1$         &  &      &      &     &     &     &     &     &     &      &      &      &      &      &      &     \\ 
				1                & $1$         &  &      &      &     &     &     &     &     &     &      &      &      &      &      &      &     \\ 
			    2                  & $1$         & $1$ &      &      &     &     &     &     &     &     &      &      &      &      &      &      &     \\ 
			    3                  & $1$         & $2$ &   $1$   &      &     &     &     &     &     &     &      &      &      &      &      &      &     \\ 
			    4                  & $1$         & $3$ &   $3$   &   $1$    &   $1$    &    &     &     &     &     &      &      &      &      &   &  &     \\
			    5                  & $1$         & $4$ &   $6$   &   $4$    &   $3$    & $2$    &  $1$   &     &     &     &   &      &   &    &   &  &     \\
			    6                  & $1$         & $5$ &   $10$   &   $10$    &   $8$    & $7$   & $5$    & $3$    & $1$    & $1$  &  &   &   &  &   &  &     \\ 
			    7               & $1$    & $6$ &   $15$   &   $20$    &   $19$    & $18$   & $16$    & $12$    & $8$    & $6$  & $3$  & $2$  & $1$&  &   &  &     \\ 
			    8  & $1$    & $7$ &   $21$   &   $35$    &   $40$    & $41$   & $41$  & $36$  & $29$  & $23$  & $18$  & $12$  & $9$& $5$  & $3$   & $1$ & $1$     \\  
				\hline
			\end{tabular}
		\end{footnotesize}
		\caption{First terms of the sequence $A129181$.  } 
	\end{table}

Of special interest is the set of Motzkin paths that never touch the $x$-axis in intermediate steps or the set of Motzkin paths that do no have horizontal steps on the $x$-axis .
\begin{definition}
A \textit{return} of a Motzkin path $p \in \mathcal{M}_n$ is a point of $p$ different from $(0, 0)$ and $(n,0)$ and belonging to the $x$-axis. We say that $p$ is an \textit{irreducible} or\textit{ elevated} Motzkin path if it does not have returns. We denote by $\mathcal{E}_n$ the set of elevated Motzkin paths in $\mathcal{M}_n$. Its cardinality is the integer $E_n$ given by
\begin{equation}\label{eq: elevated Motzkin numbers}
	E_0=E_1=0,  \quad E_n=M_{n-2}, \quad n\geq 2.
\end{equation}

\end{definition}

\begin{example}
	The lattice path $p=UUHUDDHD$ is an elevated Motzkin path in $\mathcal{E}_8$. 
\end{example}

In terms of its Motzkin word, a Motzkin path is elevated if the number of occurrences of $U$ is greater than the number of occurrences of $D$ in every subword $p_1\dots p_i$, for $1\leq i\leq n-1$. On the other hand, a Motzkin path with returns contains at least a subword of the form either $DU$ or $DH$.

\begin{definition}
A Motzkin path of length $n$ without horizontal steps on the $x$-axis will be called a \textit{Riordan path} of length $n$. The set of Riordan paths of length $n$ is denoted by  $\mathcal{R}_n$. The cardinality of $\mathcal{R}_n$ is given by the \textit{Riordan numbers} $R_n$. In \cite{Sloane} these numbers correspond to the sequence $A005043$.
\end{definition}

\begin{remark}
In the literature related to Motzkin paths there are many papers that deal with the problem of counting variations of them in terms of the appearance of different elements such as returns, horizontal steps, peaks (subwords of the form $UD$) and valleys (subwords of the form $DU$), for instance. In Section \ref{sec: the bijection} we will consider some of them (elevated Motzkin paths and Riordan paths) as translation of particular situations coming from the context of flag codes.  
\end{remark}

	\subsection{Subspace codes and flag codes}
	
Consider a prime power  $q$ and a positive integer $n\geq 2$. Let $\bbF_q$ be the finite field with $q$ elements. The set of subspaces of the vector space $\bbF_q^n$ can be seen as a metric space by using different metrics. Among them, here we will use the so-called \emph {injection distance.}
	\begin{definition}
		The \emph{injection distance} between two subspaces $\cU, \cV \subseteq \bbF_q^n$ is defined as 
		\begin{equation}\label{eq: subspace distance}
			d_I(\cU, \cV)= \max\{\dim(\cU), \dim(\cV)\}-\dim(\cU\cap \cV). 
		\end{equation}	
		In particular, if $\cU, \cV$ have the same dimension, say $1 \leq k <n$, then we have
		\begin{equation}\label{def: subspace distance in the Grassmannian}
			d_I(\cU, \cV)= k-\dim(\cU\cap \cV)=\dim(\cU+ \cV)-k.
		\end{equation}
	\end{definition}
Hence, if for any $1 \leq k < n$ we denote by $\cG_q(k, n)$ the \emph{Grassmannian}, that is,  the set of $k$-dimensional subspaces of $\bbF_q^n$, we can consider error-correcting codes in  $\cG_q(k, n)$ as follows.
	\begin{definition}
		A \emph{constant dimension code}  $\cC$ of length $n$ and dimension $k$ is a nonempty subset of $\cG_q(k, n)$. The \emph{minimum distance} of $\cC$ is defined as 
		$$
		d_I(\cC)=\min\{ d_I(\cU, \cV) \ | \ \cU, \cV \in \cC, \ \cU \neq \cV \}
		$$	
		whenever $|\cC| \geq 2$. In case $|\cC|=1$, we put $	d_I(\cC)=0$.
	\end{definition}
Note that the minimum distance of a code $\cC$ is an integer such that
\begin{equation}\label{eq: bound subspace distance}
	0\leq d_I(\cC)\leq \left\lbrace
	\begin{array}{lll}
		k     & \text{if} & 2k\leq n,\\
		n-k & \text{if} & 2k\geq n.
	\end{array}
	\right.
\end{equation}
	
	\begin{remark}
		Another frequent metric used when working with subspace codes of $\bbF_q^n$ is the \emph{subspace distance}. It is given by
		\begin{equation}\label{eq: subspace distance ds}
			d_S(\cU, \cV)=\dim(\cU+\cV)-\dim(\cU\cap\cV). 
		\end{equation}
		Observe that, if $\cU, \cV\in\cG_q(k, n)$, then $d_S(\cU, \cV)=2(k-\dim(\cU\cap\cV))=2d_I(\cU, \cV).$ Hence, in the context of constant dimension codes, the injection distance and the subspace distance are equivalent metrics. Consult \cite{TrautRosen18} and the references therein for more information on this class of codes.
	\end{remark}
	
	Let us now consider flags on $\bbF_q^n$ with the same type, that is, sequences of nested subspaces of $\bbF_q^n$ whose list of dimensions is fixed. In this way we get \textit{flag codes (of constant type)}. This idea was first proposed in \cite{LiebNebeVaz18}. Let us recall the basics on this family of codes.
	
	\begin{definition}
		A {\em flag} $\mathcal{F}=(\mathcal{F}_1,\ldots,  \mathcal{F}_r)$ on $\mathbb{F}_{q}^n$ is a sequence of $\bbF_q$-vector subspaces of $\mathbb{F}_{q}^n$ satisfying
		$$
		\{0\}\subsetneq \mathcal{F}_1 \subsetneq \cdots \subsetneq \mathcal{F}_r \subsetneq \mathbb{F}_q^n.
		$$
		The type of $\mathcal{F}$ is the vector $(\dim(\cF_1), \dots, \dim(\cF_r))$. If it equals $(1, 2, \ldots, n-1),$ we say that $\cF$ is a {\em full flag}. We say that $\mathcal{F}_i$ is the {\em $i$-th subspace} of $\cF$.
	\end{definition}

The set of all the flags on $\mathbb{F}_{q}^n$ of a fixed type vector $(t_1, \dots, t_r)$ is said to be the \emph{flag variety} $\mathcal{F}_q( (t_1, \ldots, t_r),n)$ or the \emph{full flag variety} $\mathcal{F}_q(n)$,  if the type vector is $(1, \dots, n-1)$. The distance defined in (\ref{eq: subspace distance}) can be extended to the flag variety as follows: given two flags $\cF=(\mathcal{F}_1,\dots,  \mathcal{F}_r)$ and $\cF'=(\mathcal{F}'_1, \dots,  \mathcal{F}'_r)$ in $\mathcal{F}_q( (t_1, \ldots, t_r),n)$, the \emph{(injection) flag distance} between them is the value
	\begin{equation}\label{eq: flag distance}
		d_f(\cF,\cF')= \sum_{i=1}^r d_I(\mathcal{F}_i, \mathcal{F}'_i).  
	\end{equation}
	\begin{remark}
The subspace distance $d_S$ defined in (\ref{eq: subspace distance ds}) can also be extended to the flag variety. Given $\cF$ and $\cF'$ as above, the sum of subspace distances
		\begin{equation}\label{eq:flag distance}
		\sum_{i=1}^{r} d_S(\cF_i, \cF'_i) = 2 d_f(\cF, \cF') 
		\end{equation}
		is an equivalent distance to $d_f$. In this paper we will always work with the injection flag distance and we will simply write \textit{flag distance}.
	\end{remark}	
	
	\begin{definition}
		A \emph{flag code} of type $(t_1, \dots, t_r)$ on $\bbF_{q}^n$ is a nonempty subset $\cC\subseteq \cF_q((t_1, \dots, t_r), n)$. Its {\em minimum distance} is given by
		$$
		d_f(\cC)=\min\{d_f(\cF,\cF')\ |\ \cF,\cF'\in\cC, \ \cF\neq \cF'\}
		$$
		when $|\cC| \geq 2$. If $|\cC|=1$, we put $d_f(\cC)=0$.
	\end{definition}		
		For each dimension $t_i$ in the type vector of a flag code $\cC$, we can consider the constant dimension code in the Grassmannian $\cG_q(t_i, n)$ consisting of the set of the $i$-th subspaces of flags in $\cC$. This set is called the \emph{$i$-projected code} of $\cC$ and we denote it by $\cC_i$. It is clear that  $\vert \cC_i\vert \leq \vert \cC \vert$ for every $i=1, \dots, r$.  Moreover, if given an index  $1 \leq i \leq r$ we have that  $\vert \cC_i \vert = \vert \cC \vert$, we can deduce that for any couple of flags $\cF, \cF' \in \cC$  we have that $\cF_i \neq \cF_i'$. This fact leads to the following definition.
		
\begin{definition}
We say that two flags $\cF$ and $\cF'$ of type $(t_1, \dots, t_r)$ on $\bbF_{q}^n$ \emph{collapse} at their $i$-th subspace if $\cF_i=\cF'_i$. On the other hand, two different flags $\cF$ and $\cF'$ are said to be \emph{disjoint} if they do not collapse at any subspace, i.e., if $\cF_i\neq\cF'_i,$ for every $1\leq i\leq r$. Similarly, a flag code $\cC\subseteq \cF_q((t_1, \dots, t_r), n)$ is \textit{disjoint} if it consists of disjoint flags. In terms of the projected codes, a flag code $\cC$ is disjoint if, and only if, $|\cC_1|=\dots=|\cC_r|=|\cC|$.
\end{definition}

\section{Distance vectors: how to spread the flag distance}\label{sec: distance vectors}
In sight of the flag distance definition given in (\ref{eq: flag distance}), one realizes that a fixed distance value could possibly be obtained by adding different subspace distances combinations. This fact is reflected in the following example.
	\begin{example}\label{ex: comparing flag distance with subspace distances}
	Let $\{ e_1, e_2, e_3, e_4, e_5, e_6\}$ be the standard basis of the $\bbF_q$-vector space $\bbF_q^6$. Consider the flag code $\cC$ of type $(1,3,5)$ on $\bbF_q^6$ given by the set of flags:
	$$
	\begin{array}{ccc}
		\cF^1 & = & ( \langle e_1 \rangle, \langle e_1, e_2, e_3 \rangle, \langle e_1, e_2, e_3, e_4, e_5 \rangle),\\
		\cF^2 & = & ( \langle e_5 \rangle, \langle e_4, e_5, e_6 \rangle, \langle e_1, e_2, e_4, e_5, e_6 \rangle),\\
		\cF^3 & = & ( \langle e_6 \rangle, \langle e_4, e_5, e_6 \rangle, \langle e_2, e_3, e_4, e_5, e_6 \rangle), \\
		\cF^4 & = & ( \langle e_2 \rangle, \langle e_2, e_5, e_6 \rangle, \langle e_2, e_3, e_4, e_5, e_6 \rangle).
		
	\end{array}
	$$
\noindent Observe that it holds:
	$$
	d_f(\cC)=d_f(\cF^2, \cF^3)=1+0+1=2= 1 + 1 + 0 = d_f(\cF^3, \cF^4).
	$$
\end{example}
We can follow that, even if the pairs $\cF^2, \cF^3$ and $\cF^3, \cF^4$ are both at distance $2$, the first one presents a collapse in the second subspace, that is, $\cF_2^2=\cF_2^3$, whereas the second pair collapses at the third one. As a consequence, to totally capture the relative position of two flags, it is necessary to provide more precise information beyond the distance in absolute terms. In \cite{cotas} the authors deal with this question by defining \emph{distance vectors}. Let us recall this definition.
\begin{definition}\label{def: distance vector associated to a pair of flags}
	Given two different flags $\cF, \cF'$ of type $t=(t_1, \dots,t_r)$ on $\bbF_q^n$, their associated \emph{distance vector} is
	$$
	\textbf{d}(\cF, \cF')=(d_I(\cF_1, \cF_1'), \dots, d_I(\cF_r, \cF_r')) \in \bbZ^r.
	$$
\end{definition}
Notice that the sum of the components of $\textbf{d}(\cF, \cF')$ is the flag distance $d_f(\cF, \cF')$ defined in (\ref{eq: flag distance}). Given a positive integer $n\geq 2$ and a type vector $t=(t_1, \dots,t_r)$, we denote by $D^{(t,n)}$ the maximum possible value of the flag distance in $\cF_q(t,n)$ that, as a consequence of (\ref{eq: bound subspace distance}), is
\begin{equation}\label{eq: D^(type,n)}
	D^{(t,n)}=\left(\sum_{t_i \leq \left\lfloor \frac{n}{2}\right\rfloor} t_i + \sum_{t_i > \left\lfloor \frac{n}{2}\right\rfloor} (n-t_i)\right).
\end{equation}
In particular,  the set $\{0,1,\dots,  D^{(t,n)}\}$ contains all the possible values for the flag distance in $\cF_q(t,n)$. When working with the full type vector, we simply write
\begin{equation}\label{eq: D^n}
	D^n= \left\lfloor \frac{n^2}{4}  \right\rfloor 
	=\left\lbrace
	\begin{array}{cccl}
		\frac{n^2}{4}   & \text{if} & n &  \text{is even}, \\
		& & & \\ [-1em]
		\frac{n^2-1}{4} & \text{if} & n &  \text{is odd}
	\end{array}
	\right.
\end{equation} 
to denote the maximum possible distance between full flags on $\bbF_q^n$.

\begin{definition}
	Let $d$ be an  integer such that $0\leq d\leq D^{(t,n)}$. We define \emph{the set of distance vectors associated to $d$ for the flag variety $\cF_q(t, n)$} as 
	$$
	\mathcal{D}(d, t, n) = \{ \textbf{d}(\cF, \cF') \ | \  \cF, \cF'\in\cF_q(t, n), \ d_f(\cF, \cF')= d  \} \subseteq \bbZ^r.
	$$
	On the other hand, the set of \emph{distance vectors for the flag variety $\cF_q(t, n)$} is 
	$$
	\mathcal{D}(t, n) = \{ \textbf{d}(\cF, \cF') \ | \ \cF, \cF'\in\cF_q(t, n)\} \subseteq \bbZ^r.
	$$
	and it holds
	$$
	\mathcal{D}(t, n) = \bigcup_{d} \mathcal{D}(d, t, n),
	$$
	where $d$ takes all the integers between $0\leq d\leq D^{(t, n)}$. When working with the full flag variety, we drop the type vector and simply write  $\mathcal{D}(d, n)$ and $\mathcal{D}(n),$ respectively.
\end{definition}
At this point we propose several important questions: how many possible distance vectors could correspond to a given couple of arbitrary flags $\cF$, $\cF'$ on $\bbF_q^n$, that is, what is the cardinality of $\mathcal{D}(n)$? In particular, what happens if these flags do not  collapse? What if they never share two consecutive subspaces? What if $\cF$, $\cF'$ belong to a flag code with prescribed minimum distance? We address these questions in the following section.

\section{The bijection} \label{sec: the bijection}
	
In the remain of the paper we will always work with full flag codes. More precisely, we will show that the number of distance vectors associated to the full flag variety $\cF_q(n)$ is given by the Motzkin number $M_n$. 

Let us first recall those properties that characterize distance vectors. The following result is based on \cite[Th. 3.8]{Combinatorial} and \cite[Th. 3.9]{cotas}. For completeness here we include a shorter proof adapted to the full type case.

\begin{theorem}\label{prop: allowed pattern}
	  Consider integers $\delta_1, \dots, \delta_{n-1}\geq 0$ and put $\delta_0=\delta_n=0$. 
	 Then $(\delta_1, \dots, \delta_{n-1})$ is a distance vector in $\cD(n)$ if, and only if,
\begin{equation}\label{eq: dist vect char}
	 \delta_{i} \in \{\delta_{i-1} -1, \delta_{i-1}, \delta_{i-1}+1\},
\end{equation}		 
	 for all $1\leq i\leq n$.   
\end{theorem}
\begin{proof}

Suppose that $(\delta_1, \dots, \delta_{n-1})$ is a distance vector. Then, there exists a pair of full flags $\cF, \cF'\in\cF_q(n)$ satisfying $\delta_j=d_I(\cF_j, \cF'_j)$, for every $1\leq j\leq n-1$. Let us see that
$\delta_{i} \in \{\delta_{i-1} -1, \delta_{i-1}, \delta_{i-1}+1\}.$
Note that $\delta_1, \delta_{n-1}\in\{0, 1\}$ given that $\delta_1, \delta_{n-1}$ are, respectively, distances between lines and hyperplanes of $\bbF_q^n$. Moreover, since $\delta_0=\delta_n=0$, the stated condition holds for both $i=1, n$. Consider now any index $2\leq i\leq n-1$. Note that
$$
\dim(\cF_{i-1} + \cF'_{i-1}) \leq \dim(\cF_i + \cF'_i) \leq \dim(\cF_{i-1} + \cF'_{i-1})+2.
$$
Hence, $\dim(\cF_i +\cF'_i)= \dim(\cF_{i-1} + \cF'_{i-1})+ k,$ with $k\in\{0, 1, 2\}$, and  we have
$$
\begin{array}{rcl}
\delta_i = d_I(\cF_i, \cF'_i) & = &  \dim(\cF_i+\cF_i)- i \\
                              & = & \dim(\cF_{i-1} + \cF'_{i-1})+ k-i\\
                              & = & \dim(\cF_{i-1} + \cF'_{i-1}) -(i-1) + (k-1)    \\
                              & = & d_I(\cF_{i-1} ,\cF'_{i-1}) + (k-1)\\
                               & = & \delta_{i-1} + (k-1),
\end{array}
$$
where $k-1\in\{-1, 0, 1\}.$ 

For the converse, assume that $(\delta_1, \dots, \delta_{n-1})$ is a vector satisfying (\ref{eq: dist vect char}). This condition, along with the fact that $\delta_0=\delta_n=0$, implies $\delta_i \leq \min\{i, n-i\}$, i.e., every $\delta_i$ is an admissible value for the injection distance between $i$-dimensional subspaces of $\bbF_q^n$. Using induction, we build flags $\cF, \cF'$ in $\cF_q(n)$ with $\bd(\cF, \cF')=(\delta_1, \dots, \delta_{n-1})$. First of all, by means of (\ref{eq: dist vect char}), it is clear that $\delta_1\in\{0, 1\}$. If $\delta_1=0$, just take $\cF_1=\cF'_1$. Otherwise, consider any two different lines $\cF_1, \cF'_1$. Assume now that for some $1\leq i < n-1$, we have found nested subspaces $\cF_1  \subsetneq  \dots  \subsetneq  \cF_i $ and $\cF'_1  \subsetneq  \dots  \subsetneq  \cF'_i$ such that $\delta_j=d_I(\cF_j, \cF'_j)=\dim(\cF_j+\cF'_j)-j,$ for every $1\leq j\leq i$. Let us give convenient $u, u' \in \bbF_q^n$ and subspaces $\cF_{i+1}=\cF_i+ \left\langle u\right\rangle$, $\cF'_{i+1} =\cF'_i + \left\langle u' \right\rangle$ such that $\delta_{i+1}= d_I(\cF_{i+1}, \cF'_{i+1})$ by distinguishing three situations.
\begin{itemize}
	\item $ \boldsymbol{\delta_{i+1}=\delta_i}$. In this case, notice that 
$
\dim(\cF_i+\cF'_i)-i = \delta_i = \delta_{i+1} \leq  n-(i+1).
$
Thus, $\dim(\cF_i+\cF'_i)\leq n-1$ and we can choose $u\in\bbF_q^n\setminus(\cF_i+\cF'_i)$ to form $\cF_{i+1}=\cF_i+ \left\langle u\right\rangle$. Now, if $\delta_i \neq 0$, we can take $u'\in (\cF_i+\cF'_i)\setminus \cF'_i$ and $\cF'_{i+1} =\cF'_i + \left\langle u' \right\rangle$. In case $\delta_i=0$, then $\delta_{i+1}=0$ and we just put $\cF'_{i+1}=\cF_{i+1}=\cF_i+ \left\langle u\right\rangle$. Note that, in any case, $\cF_{i+1}, \cF'_{i+1}$ are subspaces of dimension $i+1$ with
$$
\dim(\cF_{i+1}+\cF'_{i+1}) = \dim(\cF_i+\cF'_i + \left\langle u\right\rangle) = \dim(\cF_i+\cF'_i) + 1 = \delta_i + i + 1.
$$
As a consequence,  $\delta_{i+1}=\delta_i= \dim(\cF_{i+1}+\cF'_{i+1}) - (i+1)$ as desired.
\item $ \boldsymbol{\delta_{i+1}=\delta_i+1}$. Hence,
$
\dim(\cF_i+\cF'_i)-i = \delta_i = \delta_{i+1} - 1 \leq  n-(i+1)-1 = n-i-2
$
and $\dim(\cF_i+\cF'_i)\leq n-2$. This allows us to choose two linearly independent vectors $u, u'\in\bbF_q^n$ such that $$\dim(\cF_i+\cF'_i+ \langle u\rangle + \langle u'\rangle) = \dim(\cF_i+ \cF'_i)+2.$$ The subspaces $\cF_{i+1}=\cF_i+\langle u\rangle$ and $\cF'_{i+1}=\cF'_i+\langle u'\rangle$ satisfy the required condition.
\item  $ \boldsymbol{\delta_{i+1}=\delta_i-1}$. In this situation, notice that $\delta_i \geq 1$ and then $\cF_i\neq\cF'_i$. Therefore, we can find nonzero vectors $u\in\cF'_i\setminus\cF_i$ and $u'\in\cF_i\setminus\cF'_i$ to form two  $(i+1)$-dimensional subspaces $\cF_{i+1}=\cF_i+\langle u\rangle$ and $\cF'_{i+1}= \cF'_i+\langle u'\rangle$ with $d_I(\cF_{i+1}, \cF'_{i+1}) = \delta_{i+1}$.

\end{itemize}

\end{proof}

The previous theorem turns out to be the key to establish our bijection between the set of distance vectors corresponding to the full flag variety $\cF_q(n)$ and the set of Motzkin paths of length $n$.
 
\begin{theorem}\label{th: the bijection} Given a positive integer $n \geq 2$, there is a bijection between the set of distance vectors $\mathcal{D}(n)$  and the set of Motzkin paths $\mathcal{M}_n$.
\end{theorem}
\begin{proof} Recall that in this paper we do not distinguish between a Motzkin path and the corresponding Motzkin word.  Consider $(\delta_1, \ldots, \delta_{n-1}) \in \mathcal{D}(n) $ and put $\delta_0=\delta_n=0$. In light of Theorem \ref{prop: allowed pattern}, for any $i \in \{1,\dots, n\}$, it holds
	$$\delta_{i}-\delta_{i-1} \in \{-1,0,1\}.$$	 
Consider the following map:
\begin{equation}\label{eq:Bijection distance vectors Motzkin paths}
\begin{array}{cccc}
	\Psi: & \mathcal{D}(n)                  & \longrightarrow & \mathcal{M}_n \\
	      & (\delta_1, \dots, \delta_{n-1}) & \longmapsto     & p_1\dots p_n ,
\end{array}
\end{equation}
where, for every $1\leq i\leq n$, we take
\begin{equation}\label{eq:Expression of Psi}
	p_i=\left\{
	\begin{array}{cccc}
		U &  \text{ if } & \delta_i-\delta_{i-1}=&1, \\
		H &  \text{ if } & \delta_i-\delta_{i-1}=&0, \\
		D &  \text{ if } & \delta_i-\delta_{i-1}=&-1.
	\end{array}
\right.
\end{equation}

Let us see that the word $\Psi((\delta_1, \dots, \delta_{n-1}))$ is a Motzkin word. To do so, assume that, for some $1\leq i\leq n$, in the subword $p_1\dots p_i$, the letter $D$ appears more times than $U$. This means that, up to the $i$-th step, the difference $\delta_j-\delta_{j-1}=-1$ occurs more times than $\delta_j-\delta_{j-1}=1$. As a result, we have that
$
\delta_i =  \sum_{j=1}^i \delta_{j}-\delta_{j-1} < 0, 
$
which is a contradiction. Similarly, since $\sum_{j=1}^{n} \delta_{j}-\delta_{j-1}= \delta_{n}-\delta_0=0$, the number of occurrences of both $U$ and $D$ coincides.

On the other hand, given $p \in \mathcal{M}_n$, we have that $p=p_1 \dots p_n$ is a sequence of $n$ letters $p_i \in \{U, H, D\}$, for $i=1, \dots, n$ such that, in particular, $p_1 \in \{U, H\}$ and $p_n \in \{ H, D\}$. Now, for $i=1, \dots, n$, we consider the following associated values: $u_i$ (resp. $h_i$, $d_i$) is the number of occurrences of an $U$ (resp. $H$, $D$) step at or before $p_i$. Note that $d_1=0$ and $u_n=d_n$. Moreover, for every $1\leq i\leq n$, it holds $u_i+h_i+d_i=i$ and $u_i\geq d_i$. 
From these values, we define a map $\Phi: \mathcal{M}_n \rightarrow \mathcal{D}(n)$ as follows:
\begin{equation}\label{eq: expression of Phi}
	\Phi(p_1\dots p_n)=(u_1-d_1, u_2-d_2, \dots, u_{n-1}-d_{n-1}).
\end{equation}
As a consequence of the definition of $u_i, h_i, d_i$, we have that 

$$
(u_i-d_i)-(u_{i-1}-d_{i-1})
=
\left\lbrace
\begin{array}{rcl}
1  & \text{if} & p_i= U,\\
0  & \text{if} & p_i= H,\\
-1 & \text{if} & p_i= D,
\end{array}
\right.
$$
with $2\leq i \leq n$. As $u_1-d_1 \in \{0,1\}$, by means of Theorem \ref{prop: allowed pattern}, $\Phi(p_1\dots p_n)$ is a distance vector.
 
Finally, we prove that $\Phi$ and $\Psi$ are mutually inverse. To do so, consider $(\delta_1, \dots, \delta_{n-1}) \in \cD(n)$. Let us see that, if $ \Phi \circ \Psi ((\delta_1, \dots, \delta_{n-1}))= \Phi (p_1p_2 \dots p_n)=(u_1-d_1, u_2-d_2, \dots, u_{n-1}-d_{n-1})$ (defined as in (\ref{eq: expression of Phi}) and (\ref{eq:Expression of Psi})), then $\delta_i=u_i-d_i$, for all $i \in \{1,\dots, n-1\}$.  We work by induction on $n$. Clearly, if $\delta_1=1$ (resp. $\delta_1=0$), then $p_1=U$ (resp.  $p_1=H$) and $u_1-d_1=1=\delta_1$ (resp. $u_1-d_1=0=\delta_1$ ). Assume that the result is true,  for $ k < n-1$, that is, $\delta_k=u_k-d_k$.  We now that $\delta_{k+1}-\delta_k \in \{-1,0,1\}$.  Suppose, for instance, that  $\delta_{k+1}-\delta_{k}=1$.  Then $p_{k+1}=U$ and, as $\delta_k=u_k-d_k$ by induction hypothesis, we have $u_{k+1}-d_{k+1}=(u_k-d_k)+1=\delta_k+1$. Thus, $u_{k+1}-d_{k+1}=\delta_{k+1}$ as desired. The cases where $\delta_{k+1}-\delta_{k}=0$ or $\delta_{k+1}-\delta_{k}=-1$ are analogous.

Take now $p_1p_2\dots p_n$ a Motzkin word of length $n$ such that $\Psi \circ \Phi (p_1p_2\dots p_n)=p_1'p_2'\dots p_n'$. Let us see that $p_i=p_i'$, for any $i \in \{1, \dots, n\}$. Let us work again by induction on $n$. It is clear that, if $p_1=U$ ( resp. $p_1=H$) then $u_1-d_1=1$ (resp. $u_1-d_1=0$) and $p_1'=U=p_1$ (resp. $p_1'=H=p_1$). Assume that $p_k=p_k'$ for $k <n$. If $p_{k+1}'=U$, then $u_{k+1}-d_{k+1}-(u_k-d_k)=1$ which means that $p_{k+1}=U=p_{k+1}'$. The cases $p_{k+1}'=H$ or $p_{k+1}'=D$ go analogously. Hence $\Psi$ is a bijection as we wanted to show. 
\end{proof}

\begin{example}
	To the distance vector $v=(1,2,2,3,2,1,1,0) $ it corresponds the Motzkin path $\Psi(v)=UUHUDDHDH$ represented in Figure \ref{Fig:Motzkin path 1}. For it, we have $u_6=3, h_6=1, d_6=2$, for instance. In fact,
	$\Phi \circ \Psi (v)=(1-0,2-0, 2-0, 3-0, 3-1, 3-2, 3-2, 3-3)=v$.
\end{example}

We can now provide the answer to the first question proposed in Section \ref{sec: distance vectors}.
	\begin{corollary}
		The number of distance vectors in $\mathcal{D}(n)$ is the $n$-th Motzkin number $M_n$.
	\end{corollary}
\begin{remark}  Note that in our work \cite{Combinatorial}, given a pair of full flags $\cF$ and $\cF'$ on $\bbF_q^n$, we defined their \emph{distance path} $\Gamma(\cF, \cF')$ as the  polygonal path in $\mathbb{Z}^2$ from $(0,0)$ to $(n,0)$ whose intermediate vertices are the points $(i, d_I(\cF_i, \cF'_i))$ for every $0< i<n$. Therefore, by virtue of Theorem  \ref{prop: allowed pattern}, our distance paths are Motzkin paths and conversely.
\end{remark}

\subsection{Distance vectors corresponding to disjoint flags}\label{Subsec: the disjoint case}

Recall from Section \ref{sec: preliminaries} that a couple of flags $\cF, \cF'$ on $\bbF_q^n$ are disjoint if they do not share any subspace, that is, if they do not present collapses at any dimension. On the other hand, if $\cF, \cF'$ collapse at dimension $i$, that is $\cF_i=\cF_i'$, then the $i$-th component in $\textbf{d}(\cF, \cF')$ is equal to zero.

 By means  of (\ref{eq: expression of Phi}), for a distance vector $v=(\delta_1, \dots, \delta_{n-1}) \in \mathcal{D}(n)$ such that $\delta_i=0$, we have that $u_i=d_i$. Graphically, the Motzkin path $\Psi(d)$ intersects the $x$-axis at the point $(i,0)$. In other words, two full flags collapsing at dimension $i$ induce a Motzkin path in which the point $(i,0)$ is a return. This fact enables us to state the following result. Its proof can be easily obtained by using the maps $\Psi$ and $\Phi$ defined in the proof of Theorem \ref{th: the bijection}, which are still bijections when restricted to $\cE_n$ and the set of distance vectors in $\cD(n)$ without null components.

\begin{proposition}\label{prop:elevated-disjoint}
Given a positive integer $n \geq 2$, there is a bijection between the set of distance vectors in $\mathcal{D}(n)$ corresponding to pairs of disjoint flags and the set of elevated Motzkin paths $\mathcal{E}_n$.
\end{proposition}

\begin{corollary}
The number of distance vectors in $\mathcal{D}(n)$ corresponding to pairs disjoint flags is the $n$-th  number $E_n$ defined in (\ref{eq: elevated Motzkin numbers}).
\end{corollary}

\subsection{Flag distance equals area}
Let us show that the map $\Psi$ takes flag distance to area. In other words, if $v=(\delta_1, \dots, \delta_{n-1}) \in \mathcal{D}(d,n)$, that is,  $d=\delta_1+\cdots+\delta_{n-1}$, then $A(\Psi(v))=d$. To do this, we work first with elevated Motzkin paths, which means that our corresponding distance vectors have no null component, by Proposition \ref{prop:elevated-disjoint}. Let us describe a decomposition of any such a distance vector that may result very convenient for our purposes.

\begin{proposition}\label{prop:decompostion distance vector}
	Consider $v=(\delta_1, \dots, \delta_{n-1}) \in \mathcal{D}(n)$ with $\delta_i \neq 0$ for every $i \in \{1,\dots, n-1\}$. Then, there is $r\geq 1$ a positive integer  and  vectors $v^k=(v_1^k,\dots, v_{n-1}^k) \in \bbF_2^{n-1}$ for $k\in \{1, \dots, r\}$ such that 
	\begin{equation}\label{eq:suma vertical}
	\delta_i=\sum_{k=1}^{r}v_i^k.
	\end{equation}

\end{proposition}
\begin{proof}
	Take $v=(\delta_1, \dots, \delta_{n-1}) \in \mathcal{D}(n)$ with $\delta_i \neq 0$ for $i \in \{1,\dots, n-1\}$. Note that, in this case, $\delta_1=\delta_{n-1}=1$. Take $r=\max\{\delta_1, \dots, \delta_{n-1}\}$ and, for any $k\in\{1, \dots, r\}$ define a vector
	$v_k=(v_1^k, \dots, v_{n-1}^k)$ such that $v_i^k=0$ if $k> \delta_i$ and  $v_i^k=1$ otherwise. With this choice the desired equality (\ref{eq:suma vertical})  holds.
	
\end{proof}
\begin{remark}
	Note that, if we denote $\rho_k= \sum_{i=1}^{n-1}v_i^k$, from the previous decomposition, we have that 
\begin{equation}\label{eq: suma horizontal y vertical}
d= \sum_{i=1}^{n-1} \delta_i=\sum_{i=1}^{n-1}\sum_{k=1}^{r}v_i^k= \sum_{k=1}^{r}\sum_{i=1}^{n-1}v_i^k= \sum_{k=1}^{r} \rho_k.
\end{equation}
\end{remark}

\begin{example}\label{ex: vertical decomposition distance vector} Take the distance vector $v=(1,2,3,2,1,2,2,1) \in \mathcal{D}(9)$. Thus, $r=3$, $v_1=(1,1,1,1,1,1,1,1), v_2=(0,1,1,1,0,1,1,0), v_3=(0,0,1,0,0,0,0,0)$. In fact, we can put
$$
\begin{array}{cccccccccl}
	0&0&1&0&0&0&0&0& \rightarrow & \rho_3= 1\\
	0&1&1&1&0&1&1&0& \rightarrow & \rho_2=3+2\\
	1&1&1&1&1&1&1&1& \rightarrow & \rho_1=9\\
	\downarrow&\downarrow&\downarrow&\downarrow&\downarrow&\downarrow&\downarrow&\downarrow&  & \\
	1&2&3&2&1&2&2&1&  & \textbf{14}\\
	\verteq & \verteq & \verteq & \verteq & \verteq & \verteq & \verteq & \verteq & \\
	\delta_1 & \delta_2 & \delta_3 & \delta_4 & \delta_5 & \delta_6 & \delta_7 & \delta_8 & 
\end{array}
$$
and observe how the flag distance, $d=14$ in this case, can be decomposed ``vertically'' and ``horizontally''.

\end{example}
Following the ideas in \cite{Guay-Petersen}, let us see that the area under a Motzkin path $p \in \mathcal{M}_p$ can be decomposed into horizontal strips determined by the up and down steps. 
\begin{proposition}\label{prop:formula suma horizonatl strips}
Let $p \in  \mathcal{E}_n$ be an elevated Motzkin path. Then
\begin{equation}\label{eq:formula suma horizonatl strips}
	A(p)=\sum_{p_j=D} j- \sum_{p_i=U} i.
\end{equation}
\end{proposition}
\begin{proof}
Given $p=p_1p_2\dots p_n \in \mathcal{M}_n$, as it was pointed out in Remark \ref{rem: parenthesis}, the letters $U$ and $D$ in $p$ form a balanced  parenthesization, so they can be naturally paired up. Moreover, every couple of matched steps $p_i=U$ and $p_j=D$ determines a horizontal strip of area $j-i$ below the Motzkin path $p$. The sum of areas of all the corresponding horizontal strips gives us expression (\ref{eq:formula suma horizonatl strips}) for $A(p)$.

\end{proof}
 \begin{example} Let us take $p \in \mathcal{M}_9$ with $p=UUUDDUHDD$. In this case, we have the paired steps $(p_1, p_9)$, $(p_2,  p_5)$, $(p_6,p_8)$ and $(p_3, p_4)$.
 		\begin{figure}[H]
\begin{center}
			\includegraphics[scale=0.75]{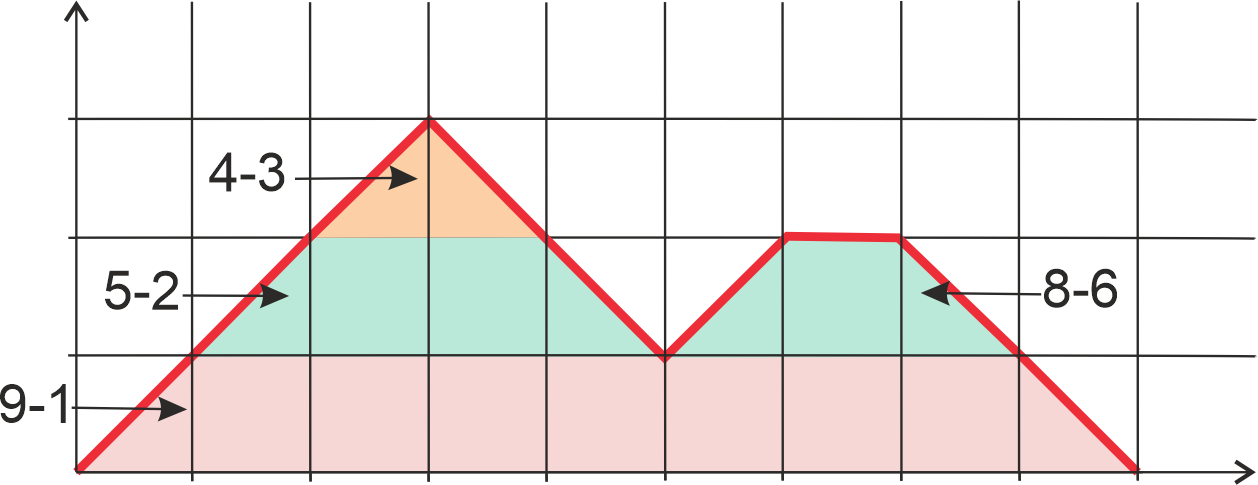}
			\caption{The Motzkin path $p=UUUDDUHDD$ in $\mathcal{M}_9$.}
			\label{Fig:Motzkin path 2}	
\end{center}
 	\end{figure}
Then we have that $$A(p)=(9-1)+(5-2)+(8-6)+(4-3)=8+3+2+1=14.$$

Observe that our path $p$ satisfies $p=\Psi(v)$ where $v$ is the vector considered in Example \ref{ex: vertical decomposition distance vector}. Moreover, the areas corresponding to horizontal strips at height $1$, $2$ or $3$, as showed in the picture above, correspond, respectively, to the sum of the non-zero components of $v_1$, $v_2$ or $v_3$. 
 \end{example} 
 
\begin{remark}\label{rem: generalization area}
	Note that the null distance vector in $\mathcal{D}(n)$ goes by $\Psi$ to the ``flat'' Motzkin path. Besides, a null component in a distance vector $v$ does not contribute to the computation of the corresponding value of the flag distance and gives a return on $\Psi(v)$. On the other hand, a Motzkin path $p$ with a finite number of returns can be decomposed into a finite number of elevated Motzkin paths whose sum of areas gives $A(p)$. In Figure \ref{Fig:Motzkin path 3} we can see a Motzking path $p \in \mathcal{M}_9$ with a return in $(4,0)$. This path corresponds to the distance vector $v=(1,2,1,0,1,2,2,1) \in \mathcal{D}(10, 9)$. Then $p$ can be decomposed into two elevated Motzkin paths $p'=UUDD \in \mathcal{M}_4$ and $p''= UUHDD \in \mathcal{M}_5$ such that $A(p)=A(p')+A(p'')=10$.
	
\end{remark}

	\begin{figure}[H]
	\begin{center}

		\includegraphics[scale=0.75]{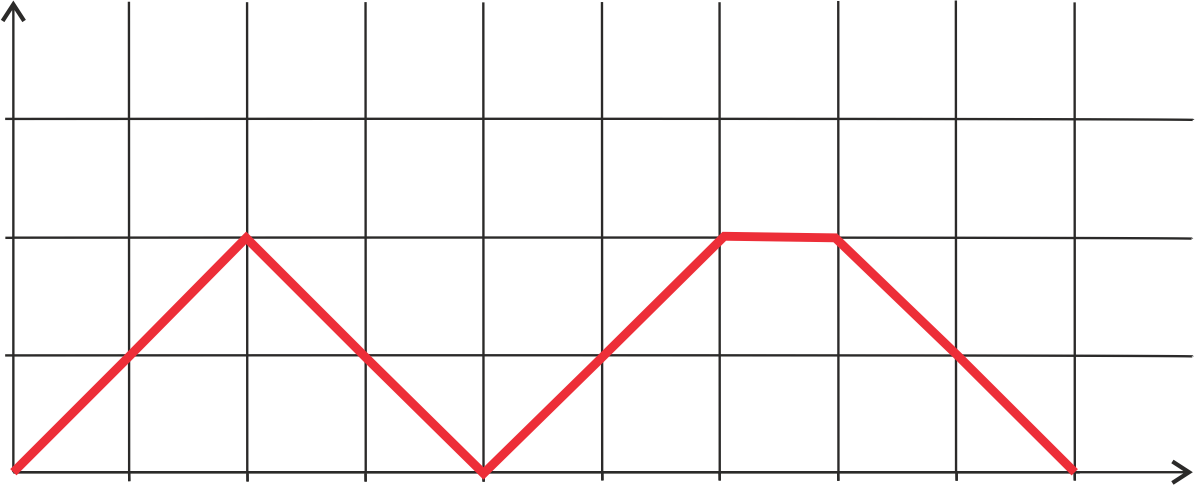}
		\caption{The Motzkin path $p=UUDDUUHDD$ in $\mathcal{M}_9$.}
		\label{Fig:Motzkin path 3}
	\end{center}
\end{figure}

Proposition \ref{prop:decompostion distance vector} along with Proposition \ref{prop:formula suma horizonatl strips} and Remark \ref{rem: generalization area}, give straightforwardly the following result.
\begin{theorem}\label{th:distance equals area}
	Consider the map $\Psi:\mathcal{D}(n) \rightarrow \mathcal{M}_n$ defined in Theorem \ref{th: the bijection}. Given a distance vector $v=(\delta_1 \dots, \delta_{n-1}) \in \mathcal{D}(n)$, we have that $A(\Psi(v))=\delta_1+\cdots+\delta_{n-1}$. In other words, the map $\Psi$ takes flag distance to area and the restriction $\Psi\mid_{\mathcal{D}(d,n)}:  \mathcal{D}(d,n) \rightarrow \mathcal{M}_n(d)$ is a bijection for any integer $d \in [0, \lfloor \frac{n^2}{4} \rfloor].$
\end{theorem}

\begin{corollary}\label{cor:distance area sequence}
	The number of distance vectors in $\mathcal{D}(n)$ corresponding to a fixed flag distance value $d$ equals the number of Motzkin paths $p \in \mathcal{M}_n$ such that $A(p)=d$.  This number is the term $T(n,d)$  of the sequence $A129181 $ in \cite{Sloane}.
\end{corollary}

Observe that, as a consequence of Theorem \ref{th:distance equals area}, we could directly deduce the bound given in (\ref{eq: D^n}) and, by Theorem \ref{prop: allowed pattern}, conclude that it is sharp. It is enough to see that $\lfloor \frac{n^2}{4} \rfloor$ is the largest possible area under a Motzkin path with $n$ steps. This area is attained by the path $p=p_1p_2 \dots p_n$ with $p_1=\dots=p_k=U$ and $p_{k+1}= \dots =p_n=D$, if  $n$ is even ($n= 2k, k\geq 1)$, or with $p_1=\cdots=p_k=U$, $p_{k+1}=H$ and $p_{k+2}= \dots =p_n=D$,  if $n$ is odd ($n= 2k+ 1, k\geq 1)$.
\begin{remark}
	An analogous result to Theorem \ref{th:distance equals area} was proved also in \cite{Combinatorial}. In that paper, the authors define distance paths in a distance support and show that the area under a distance path corresponds to its associated value of the flag distance by counting lattice points in the path, or below it, and applying  Pick's Theorem.
\end{remark}

\subsection{The number of possible distance vectors of a flag code}
As a consequence of Theorem \ref{th:distance equals area}, we can easily compute the number of possible distance vectors that can be potentially associated to a full flag code with prescribed minimum distance.

\begin{definition}\label{def: distance vectors of a flag code}
	Given a full flag code $\cC\subseteq\cF_q(n)$,  its \emph{set of distance vectors} is defined as 
	$$
	\mathcal{D}(\cC) = \{ \textbf{d}(\cF, \cF') \ | \ \cF, \cF'\in\cC, \ d_f(\cF, \cF')=d_f(\cC)\}.
	$$
\end{definition}

\begin{example}\label{ex: dv d=2}
	Let $\{e_1, e_2, e_3, e_4\}$ be the standard basis of $\bbF_q^4$ as $\bbF_q$-vector space and consider the full flag code $\cC$ on $\bbF_q^4$ given by
	$$
	\begin{array}{ccccc} 
		\mathcal{F}^1 &=& (\left\langle e_1 \right\rangle, & \left\langle e_1, e_2 \right\rangle , & \left\langle e_1, e_2, e_4 \right\rangle),\\
		\mathcal{F}^2 &=& ( \left\langle e_1 \right\rangle , &  \left\langle e_1, e_3 \right\rangle, & \left\langle e_1, e_2, e_3 \right\rangle),\\
		\mathcal{F}^3 &=& ( \left\langle e_2 \right\rangle, & \left\langle e_2, e_3 \right\rangle , & \left\langle e_1, e_2, e_3 \right\rangle).
	\end{array}
	$$
Notice that 
	$$
	\begin{array}{ccccccccccc}
		d_f(\cF^1, \cF^2) & = & 0 &+& 1 &+& 1 & = & 2, \\
		d_f(\cF^1, \cF^3) & = & 1 &+& 1 &+& 1 & = & 3,\\
		d_f(\cF^2, \cF^3) & = & 1 &+& 1 &+& 0 & = & 2.
	\end{array}
	$$
	Hence, we have $d_f(\cC)=2$ and $\mathcal{D}(\cC)=\{ (0,1,1), (1,1,0)\}\subsetneq \mathcal{D}(2,4).$ On the other hand, there are more distance vectors in $\cD(2, 4)$. It suffices to apply Theorem \ref{prop: allowed pattern} and see that $(1,0,1)$ is an element in $\cD(2, 4).$
\end{example}
Directly from Corollary \ref{cor:distance area sequence} we can state the following results concerning full flag codes.
\begin{corollary}
	 Given a full flag code $\cC$ in $\bbF_q^n,$ the number of distance vectors in $\mathcal{D}(d_f(\cC), n)$ equals the number of Motzkin paths $p \in \mathcal{M}_n$ such that $A(p)=d_f(\cC)$, i.e., the value $T(n, d_f(\cC))$.  
\end{corollary}

\begin{remark}
Notice that the previous result counts the number of possible distance vectors associated to the minimum distance of a flag code. Nevertheless, this number might not coincide with the exact number of distance vectors of the code. For instance, if we take $\cF^1$ and $\cC$ as in Example \ref{ex: dv d=2}, it holds $d_f(\cC)=2$ and $|\cD(\cC)|= 2.$ Moreover, if we consider $\cC'=\{\cF^1, \cF^4\},$ with $$\cF^4= ( \left\langle e_2 \right\rangle, \left\langle e_1, e_2 \right\rangle , \left\langle e_1, e_2, e_3 \right\rangle),$$ then we also have $d_f(\cC')=2$ but $|\cD(\cC')|=1$. Even more, the sets $\cD(\cC)$ and $\cD(\cC')$ do not share any distance vector and form a partition of $\cD(2, 4)$, which contains $T(4, 2)=3$ elements.

On the other hand, in the special case of a flag code on $\bbF_q^n$ with minimum distance $D^n$, we can give the exact number of its distance vectors (observe that $T(n, D^n)=1$), as stated in the next result.
\end{remark}

\begin{corollary}
Let $\cC$ be a full flag code in $\bbF_q^n$ such that $d_f(\cC)=D^n$. Then $\cC$ has a unique distance vector.
\end{corollary}

We finish this section by computing the number of possible distance vectors that a disjoint full flag code can have.

\begin{proposition}
Let $\cC$ be a disjoint full flag code in  $\bbF_q^n$. Then $d_f(\cC)\geq n-1$ and the number of  possible distance vectors for $\cC$ is $T(n-2,d_f(\cC)-n+1)$.
\end{proposition}
\begin{proof} 
First of all, notice that, if $\cC$ is disjoint, then its minimum distance is attained by distance vectors $(\delta_1, \dots, \delta_{n-1})$ with no zero components, i.e., $\delta_i\geq 1$, for every $1\leq i\leq n-1$. Hence, such vectors can be decomposed into
$$
(\delta_1, \dots, \delta_{n-1}) = (1, \dots, 1) + (\gamma_0, \gamma_1, \dots, \gamma_{n-3}, \gamma_{n-2})
$$
with $\gamma_i = \delta_{i+1}-1$, for every $0\leq i\leq n-2$. Notice that $\gamma_0=\gamma_{n-2} = 0$ and, for every $1\leq i\leq n-2$, it holds
$$
\gamma_i - \gamma_{i-1} = \delta_{i+1}- \delta_{i} \in\{ -1, 0, 1\}.
$$

As a consequence, and by means of Theorem \ref{prop: allowed pattern}, $(\delta_1, \dots, \delta_{n-1})$ can be naturally identified with the distance vector $(\gamma_1, \dots, \gamma_{n-3})\in\cD(n-2)$, associated to the distance value
$$
\sum_{i=1}^{n-3} \gamma_i = \sum_{i=1}^{n-3} (\delta_{i+1}-1)  = \sum_{i=0}^{n-2} (\delta_{i+1}-1) = d_f(\cC)-n+1.
$$
Hence, the number of distance vectors in $\cD(d_f(\cC), n)$ with no zero components, that is,  the number of potential distance vectors for $\cC$, coincides with the number $T(n-2,d_f(\cC)-n+1)$.
\end{proof}

In the next table we can see the number of possible distance vectors associated to a disjoint flag code $\cC$ on $\bbF_q^n$ with prescribed minimum distance $d$ for small values of $n$. 
\begin{table}[H]\label{Tab: T(n-2,d-n+1)}
		\centering
		\begin{footnotesize}
			\begin{tabular}{c|ccccccccccccccccc}
				\hline                
				  \backslashbox{$n$}{$d$} & $\textbf{0}$       & $\textbf{1}$ &  $\textbf{2}$ &  $3$ & $\textbf{4}$ & $5$ & $\textbf{6}$ & $7$ & $8$ & $\textbf{9}$ & $10$ & $11$ & $\textbf{12}$ & $13$ & $14$ & $15$ & $\textbf{16}$    \\ 
				\hline
			    2                  & -         & $1$ &      &      &     &     &     &     &     &     &      &      &      &      &      &      &     \\ 
			    3                  & -        & - &   $1$   &      &     &     &     &     &     &     &      &      &      &      &      &      &     \\ 
			    4                  & -         & - &   -   &   $1$    &   $1$    &    &     &     &     &     &      &      &      &      &   &  &     \\
			    5                  & -         & - &   -   &   -    &   $1$    & $2$    &  $1$   &     &     &     &   &      &   &    &   &  &     \\
			    6                  & -         & - &   -   &   -    &   -    & $1$   & $3$    & $3$    & $1$    & $1$  &  &   &   &  &   &  &     \\ 
			    7               & -    & - &   -   &   -    &   -    & -   & $1$    & $4$    & $6$    & $4$  & $3$  & $2$  & $1$&  &   &  &     \\ 
			    8  & -    & - &   -   &   -    &   -    & -   & -  & $1$  & $5$  & $10$  & $10$  & $8$  & $7$& $5$  & $3$   & $1$ & $1$     \\  
				\hline
			\end{tabular}
		\end{footnotesize}
		\caption{Numbers $ T(n-2, d-n+1)$ for small values of $n$.} 
	\end{table}

\subsection{Other cases, other sequences}

As it was done in Subsection \ref{Subsec: the disjoint case} with the case of disjoint flag codes, it can be very useful to interpret the properties of a given family of flags in terms of a concrete family of Motzkin paths in order to localize a relevant sequence of integers to count the number of associated distance vectors. To this end we contemplate a last family of flags.

\begin{proposition} Let  $\mathcal{H}_n$ be the set of full flags on $\bbF_q^n$ such that given $\cF, \cF' \in \mathcal{H}_n$, they never share consecutive subspaces. Hence, the number of possible distance vectors corresponding to couples of flags in $\mathcal{H}_n$ is given by the sequence of Riordan numbers, that is, sequence $A005043$ in \cite{Sloane}.	
\end{proposition}
\begin{proof}
	It is enough to observe that a couple of flags in  $\cF, \cF' \in \mathcal{H}_n$ never presents consecutive collapses, that is, their associated distance vector $\textbf{d}(\cF,\cF')$ never has two consecutive null components. Thus, the Motzkin path $\Psi(\textbf{d}(\cF,\cF'))$ does not have horizontal steps in the $x$-axis and, as consequence, it is a Riordan path.
\end{proof}
  
\section{Conclusions and open questions}\label{sec:conclusions}
In this paper we have addressed the problem of counting the number of distance vectors associated with the full flag variety $\mathcal{F}_q(n)$. Moreover, we have provided the number of possible distance vectors for a full flag code with prescribed minimum distance either in the general case or in the disjoint one.
The key to compute these cardinalities is to associate biunivocally a Motzkin word of length $n$ with a distance vector in $\cD(n)$. 

The problem of calculating the possible number of distance vectors when we consider flags of general type $t=(t_1, \dots, t_r)$ is still open. We believe that, to deal with this question, it could be useful the use of Motzkin paths of higher rank where other kind of steps are allowed (see \cite{Mansour}).

On the other hand, taking into account the definition of the distance in the context of multishot codes (\cite{NobUcho2009, NobUcho2010}), not necessarily flag codes, we think that our work can be a starting point in the study of the possible distributions of the total distance by using lattice paths. 

\section{Acknowledgements} We are very grateful to Paulo Almeida and Alessandro Neri for suggesting that we explore the unexpected relationship between Dyck-Motzkin paths and the distance paths we have introduced in our work \cite{Combinatorial}.


\begin{thebibliography}{99}
		
		\bibitem{AhlsCai00}
		R.\ Ahlswede, N.\ Cai, R.\ Li and R.\ W.\ Yeung,  
		\emph{Network Information Flow},
		IEEE Transactions on Information Theory, Vol. 46 (2000), 1204-1216.
		
		\bibitem{Aigner1998}
		M.\ Aigner,
		\emph{Motzkin Numbers},
		European Journal of Combinatorics, Vol. 19 (1998), 663-675.
		
		\bibitem{Combinatorial}
		C.\ Alonso-Gonz\'alez and M.A.\ Navarro-P\'erez,
		\emph{A Combinatorial Approach to Flag Codes}, 
		\url{https://arxiv.org/pdf/2111.15388} (preprint).
		
		\bibitem{cotas}
		C.\ Alonso-Gonz\'alez, M.A.\ Navarro-P\'erez and X.\ Soler-Escriv\`a, 
		\emph{Flag Codes: Distance Vectors and Cardinality Bounds}, 
		\url{https://arxiv.org/abs/2111.00910} (preprint).		
		
		\bibitem{CasoPlanar}
		C.\ Alonso-Gonz\'alez, M.A.\ Navarro-P\'erez and X.\ Soler-Escriv\`a, 
		\emph{Flag Codes from Planar Spreads in Network Coding}, 
		Finite Fields and Their Applications, Vol. 68 (2020),  101745.
		
	    \bibitem{CasoNoPlanar}
	    C.\ Alonso-Gonz\'alez, M.\ A.\ Navarro-P\'erez and X.\ Soler-Escriv\`a, 
		\emph{Optimum Distance Flag Codes from Spreads via Perfect Matchings in Graphs},
		Journal of Algebraic Combinatorics, Vol. 54 (2021), 1279–1297.
		
	    \bibitem{DonaSha1977}
		R.\ Donaghey  and L. W. \ Shapiro,
		\emph{Motzkin Numbers},
		Journal of Combinatorial Theory, Series A, Vol. 23 (1977), 291-301.
		
	  \bibitem{Dukes}
		M.\ Dukes,
		\emph{The Sandpile Model on the Complete Split Graph, Motzkin Words, and Tiered Parking Functions},
		Journal of Combinatorial Theory, Series A, Vol. 180 (2021), 105418.
		
				
	 	\bibitem{Guay-Petersen}
		 M.\ Guay-Paquet and K. \ Petersen,
	 	\emph{The Generating Function for Total Displacement},
	 	The Electronic Journal of Combinatorics, Vol. 21(3) (2014), P3.37.
		
		\bibitem{KoetKschi08}
		R.\ Koetter and F.\ Kschischang,  
		\emph{Coding for Errors and Erasures in Random Network Coding},
		IEEE Transactions on Information Theory, Vol. 54 (2008), 3579-3591.
		
		\bibitem{Kurz20} 
		S.\ Kurz,
		\emph{Bounds for Flag Codes},
		Designs, Codes and Cryptography, Vol. 89 (2021), 2759–2785.
		
		\bibitem{LiebNebeVaz18}
		D.\ Liebhold, G.\ Nebe and A.\ Vázquez-Castro,  
		\emph{Network Coding with Flags},
		Designs, Codes and Cryptography, Vol. 86 (2) (2018), 269-284.
		
		\bibitem{Mansour}
		T.\ Mansour,
		\emph{Motzkin Numbers of Higher Rank:
		Generating Function and Explicit Expression}, Journal of Integer Sequences, Vol. 10 (2007), Art.7.7.4.
		
		\bibitem{Motzkin}
		T.\ Motzkin,
		\emph{Relations between Hypersurface Cross Ratios, and a Combinatorial Formula for Partitions of a Polygon, for Permanent Preponderance, and for Non-Associative Products},
		Bulletin of the American Mathematical Society, Vol. 54 (1948), 352-360.
		
		\bibitem{NobUcho2009}
		R.\ W. \ N\'obrega and B.\ F.\ Uch\^oa-Filho,
		\emph{Multishot Codes for Network Coding: Bounds and a Multilevel Construction},
		in: 2009 IEEE International Symposium on Information Theory, Proceedings (ISIT), Seoul, South Korea, 2009, pp. 428–432.
		
		\bibitem{NobUcho2010}
		R.\ W. \ N\'obrega and B.\ F.\ Uch\^oa-Filho,
		\emph{Multishot Codes for Network Coding Using Rank-Metric Codes},
		in: 2010 Third IEEE International Workshop on Wireless Network Coding, Boston, USA, 2010, pp. 1–6.
		
		\bibitem{Sloane}
		N.\ J.\ A.\ Sloane, 
		\emph{The On-Line Encyclopedia of Integer Sequences},
	    \url{https://oeis.org/}.
	    
	    \bibitem{Stanley}
	    N.\ J.\ A.\ Sloane, 
	    \emph{Ennumerative Combinatorics}, Vol. 2, Cambridge University Press, 1999.
	    		
		\bibitem{TrautRosen18} 
		A.-L.\ Trautmann and J.\ Rosenthal,
		\emph{Constructions of Constant Dimension Codes},
		in: M.\ Greferath \textit{et al.} (Eds.), Network Coding and Subspace Designs, E-Springer International Publishing AG, 2018, pp. 25-42.
		
	\end{thebibliography}
\end{document}